\documentclass{article}
 \usepackage{amsmath}
\usepackage{amssymb}
\usepackage{latexsym}
\usepackage{amsthm}
\usepackage{mathrsfs}
\usepackage{xcolor}
\usepackage{eucal} 
\usepackage{wasysym}
\usepackage{fancyhdr}
\usepackage{amsfonts}
\usepackage{amssymb} 
\usepackage{epsfig}
\usepackage{epstopdf} 
\usepackage[normalem]{ulem}

\newtheorem{theorem}{Theorem}[section]

\newtheorem{corollary}[theorem]{Corollary}

\newtheorem{lemma}[theorem]{Lemma}

\newtheorem{conjecture}[theorem]{Conjecture}
\newtheorem{observation}[theorem]{Observation}
\theoremstyle{remark}

\newcommand{\scc}{\text{\upshape{scc}}}

\title{Minimum $T$-Joins and Signed-Circuit Covering }
\author{Yezhou Wu\thanks{Key Laboratory of Ocean Observation-Imaging Testbed of Zhejiang Province, Zhejiang University, Zhoushan, Zhejiang 316021, P.R. China; Email: yezhouwu@zju.edu.cn. Partially supported by a grant of National Natural Science Foundation of China (No. 11501504) and a grant of Natural Science Foundation of Zhejiang Province (No. LY16A010005).}\; and
Dong Ye\thanks{Corresponding author. Department of Mathematical Sciences \& Center for Computational Sciences, Middle Tennessee State University, Murfreesboro, TN 37132. Email: dong.ye@mtsu.edu. Partially supported by a grant from the Simons Foundation (No. 396516).}}

\begin{document}

\maketitle

\begin{abstract}
Let $G$ be a graph and $T$ be a vertex subset of $G$ with even cardinality. A $T$-join of $G$ is a subset $J$ of edges such that a vertex of $G$ is incident with an odd number of edges in $J$ if and only if the vertex belongs to $T$. Minimum $T$-joins have many applications in combinatorial optimizations. In this paper, we show that a minimum $T$-join of a connected graph $G$ has at most $|E(G)|-\frac 1 2 |E(\widehat{\, G\,})|$ edges where $\widehat{\,G\,}$ is the maximum bidegeless subgraph of $G$.  Further, we are able to use this result to show that every flow-admissible signed graph $(G,\sigma)$ has a signed-circuit cover with length at most $\frac{19} 6 |E(G)|$. Particularly, a 2-edge-connected signed graph $(G,\sigma)$ with even negativeness has a signed-circuit cover with length at most $\frac 8 3 |E(G)|$.  
\medskip


\end{abstract}

\section{Introduction}

In this paper, a graph may have multiple edges and loops. A loop is also treated as an edge. We follow the notation from the book \cite{CQ12}. A {\em cycle} is a connected 2-regular graph. An {\em even graph} is a graph in which every vertex has even degree. An {\em Eulerian graph} is a connected even graph. In a graph $G$, a {\em circuit} is the same as a cycle, which is minimal dependent set of the graphic matroid defined on $G$. A {\em circuit cover} $\mathcal C$ of a graph is a family of circuits which cover all edges of $G$. The {\em length} of a circuit cover is defined as $\ell(\mathcal C)=\sum_{C\in \mathcal C} |E(C)|$. It is a classic optimization problem initiated by Itai et. al.  \cite{ILPR} to investigate circuit covers of graphs with shortest length. Thomassen \cite{CT} show that it is NP-complete to determine whether a bridgeless graph has a circuit cover with length at most $k$ for a given integer $k$. A well-known conjecture, the Shortest Cycle Cover Conjecture, in this area was made by Alon and Tarsi \cite{AT} as follows.  

\begin{conjecture}[Alon and Tarsi \cite{AT}]\label{conj:SCDC}
Every bridgeless graph $G$ has a circuit cover with length at most $ \frac 7 5 |E(G)|$.
\end{conjecture}

If the Shortest Cycle Cover Conjecture is true, then the bound given in the conjecture will be optimal as the shortest cycle cover of Petersen graph attains the bound. It is known that there connections of the Shortest Cycle Cover Conjecture to other problems in graph theory, such as, the Jeager's Petersen Flow Conjecture \cite{FJ} and the Circuit Double Cover Conjecture (cf. \cite{JT}) due to Seymour \cite{Se79} and Szekeres \cite{GS}. The best bound toward the Shortest Cycle Cover Conjecture is due to Bermond, Jackson and Jaeger~\cite{BJJ}, and independently due to Alon and Tarsi \cite{AT} as follows, which was further generalized by Fan~\cite{Fan90} to weighted bridgeless graphs. 

\begin{theorem}[Bermond, Jackson and Jaeger \cite{BJJ}, Alon and Tarsi \cite{AT}]\label{thm:5/3}
Every bridgeless graph $G$ has a circuit cover with length at most $ \frac 5 3 |E(G)|$.
\end{theorem}

The circuit cover problem is studied for matroids (cf. \cite{FG, PDS}). Recently, a great deal of attention has been paid to the shortest circuit cover problem of signed graphs (or signed-graphical matroids) (cf. \cite{CF,CLLZ,  KLMR,  mrrs, WY}).

 A {\em signed graph} is a graph $G$ associated with a mapping $\sigma: E(G)\to \{-1,+1\}$. A cycle of a signed graph $(G,\sigma)$ is {\em positive} if it contains an even number of negative edges, and {\em negative} otherwise. 
A {\em short barbell} is a union of two negative cycles $C_1$ and $C_2$ such that $C_1\cap C_2$ is a single vertex, and a {\em long barbell} consists of two disjoint negative cycles $C_1$ and $C_2$ joined by a minimal path $P$ (which does not contain a subpath joining $C_1$ and $C_2$). A {\em barbell} of a signed graph $(G,\sigma)$ could be a short barbell or a long barbell.
A circuit or {\em signed-circuit} of a signed graph $(G,\sigma)$ is a positive cycle or a barbell, which is a minimal dependent set in the signed graphic matroid associated with $(G,\sigma)$. A signed-circuit admits a nowhere-zero 3-flow \cite{Bo}. A signed graph with a signed-circuit cover must admits a nowhere-zero integer flow. (For definition of nowhere-zero flow, readers may refer to  \cite{Bo, CQ12}.) In fact, a signed graph has a signed-circuit cover if and only if it admits a nowhere-zero flow, so-called, flow-admissible \cite{Bo}.

It has been evident in \cite{WY} that a 3-connected signed-graph may not have a signed-circuit double cover, i.e., every edge is covered twice. This fact leads the current authors to ask, whether every flow-admissible signed graph $(G,\sigma)$ has a signed-circuit cover with length at most $2|E(G)|$? 
%For a flow-admissible signed graph $(G,\sigma)$, the length of a shortest signed-circuit cover is denoted by $\scc(G,\sigma)$. 
The first bound on the shortest circuit cover of signed graphs was obtained by  M\'a\v{c}ajov\'a et. al. \cite{mrrs} as follows.

\begin{theorem}[M\'a\v{c}ajov\'a, Raspaud, Rollov\'a and \v{S}koviera, \cite{mrrs}]
Every flow-admissible signed graph $(G,\sigma)$ has a signed-circuit cover with length at most $11|E(G)|$.
\end{theorem}

The above result has been significently improved by Cheng et. al. \cite{CLLZ}  as follows.

\begin{theorem}[Cheng, Lu, Luo and Zhang \cite{CLLZ}]
Every flow-admissible signed graph $(G,\sigma)$ has a signed-circuit cover with length at most  $\frac{14} 3 |E(G)|$.
\end{theorem}

The bound of Cheng et. al. has been further improved recently by Chen and Fan \cite{CF}. 

\begin{theorem}[Chen and Fan \cite{CF}]
Every flow-admissible signed graph $(G,\sigma)$ has a signed-circuit cover with length at most $\frac{25} 6 |E(G)|$.
\end{theorem}
 
Recently, Kaiser et. al. \cite{KLMR} announce that every flow-admissible signed graph has a signed-circuit cover with length at most $\frac{11} 3 |E(G)|$. For a 2-edge-connected cubic signed graph $(G,\sigma)$, the present authors  \cite{WY} show that $(G,\sigma)$ has a signed-circuit cover with length less than   $\frac{26} 9 |E(G)|$.

In this paper, we investigate the minimum $T$-join and the shortest signed-circuit cover problem. We show that a minimum $T$-join of a connected graph $G$ has at most $|E(G)|-\frac 1 2 |E(\widehat{\,G\,})|$ edges where $\widehat{\,G\,}$ is the maximal bridgeless subgraph of $G$. The bound is tight. Further, we are able to use the bound of the size of a minimum $T$-join to show the following result, which improves previous bounds on the length of shortest signed-circuit cover.

\begin{theorem}\label{thm:main}
Every flow-admissible signed graph $(G,\sigma)$ has a signed-circuit cover with length less than  $\frac{19}{6}|E(G)|.$ 
\end{theorem}

Particularly, if $(G,\sigma)$ is 2-edge-connected and has even negativeness (a detailed definition is given in the next section), then $(G,\sigma)$ has a signed-circuit cover with length less than $\frac 8 3 |E(G)|$.

%%%%%%%%%%%%%%%%%%%%%%%%%%%%%%
\section{Circuit covers and $T$-joins}
%%%%%%%%%%%%%%%%%%%%%%%%%%%%%%

Let $G$ be a graph. The maximal bridgeless subgraph of $G$ is denoted by $\widehat{\,G\,}$. If $G$ is a bridgeless graph, then $\widehat {\,G\,}=G$. A {\em circuit $k$-cover} of a graph $G$ is a family of circuits which covers every edge of $G$ exactly $k$ times.  

\begin{theorem}[Bermond, Jackson and Jeager \cite{BJJ}]\label{lem:7-cycle-4-cover}
Every bridgeless graph $G$ has a 4-cover by 7 even-subgraphs.
\end{theorem}

Let $G$ be a graph and $T$ be a set of vertices of $G$. A {\em $T$-join} $J$ of $G$ with respect to $T$ is a subset of edges of $G$ such that $d_J(v)\equiv 1\pmod 2$ if and only if $v\in T$, where $d_J(v)$ is the number of edges of $J$ incident with the vertex $v$. A $T$-join in a connected graph $G$ is easy to construct: partition $T$ into $|T|/2$ pairs of vertices; and, for each pair of vertices, find a path joining them; and then the symmetric difference of all edge sets of these $|T|/2$ paths generates a $T$-join of $G$. On the other hand, the subgraph induced by a $T$-join always contains $|T|/2$ paths joining $|T|/2$ pairs of vertices of $T$ which form a partition of $T$. The symmetric difference of a $T$-join and the edge set of an even-subgraph is another $T$-join of $G$.

A $T$-join is {\em minimum} if it has minimum number of edges among all $T$-joins. A minimum $T$-join $J$ induces an acyclic subgraph of $G$ (otherwise, the symmetric difference of $J$ and a cycle induced by edges of $J$ is a smaller $T$-join of $G$). 
%So a minimum $T$-join induces $|T|/2$ edge-disjoint paths. 
There is a strong polynomial-time algorithm to find a minimum $T$-join in a given graph $G$ (see \cite{EJ}). The minimum $T$-join problem and its variations have applications to many other problems in combinatorial optimization and graph theory, for example, the Chinese postman problem \cite{EJ}, Traveling Salesman Problem (TSP) \cite{CFG} and others (see  \cite{Frank, AG}). 

\begin{theorem}\label{thm:T-join}
Let $G$ be a connected graph and $T$ be a set of vertices of even size. Then a minimum $T$-join of $G$ has at most $|E(G)|-\frac 1 2 |E(\widehat{\,G\,})|$ edges. 
\end{theorem}
\begin{proof}  Let $G$ be a connected graph and $J$ be a minimum $T$-join of $G$. Note that $\widehat{\,G\,}$ is bridgeless. By Lemma~\ref{lem:7-cycle-4-cover}, $\widehat{\,G\,}$ has 7-even-subgraph  4-cover  $\{H_1,H_2,\ldots,H_7\}$. Since each $H_i$ is an even-subgraph, the symmetric difference  $H_i\oplus J$ is still a $T$-join of $G$, denoted by $J_i$ for $i\in \{1,2,\ldots,7\}$. Let $\mathcal J=\{J, J_1,\ldots, J_7\}$. 
By the choice of $J$, it holds that $|J|\leq |J_i|$ for $i\in \{1,2,\ldots,7\}$.

Let $e$ be an edge of  $\widehat{\, G\,}$. If $e\notin J$, then $e\in J_i$ if and only if $e\in E(H_i)$ for $i\in \{1,\ldots, 7\}$. Hence $e$ is covered exactly four times by the family of $T$-joins $\mathcal J$ if $e\notin J$.  If $e\in J$, then $e\in J_i$ if and only if $e\notin E(H_i)$ for $i\in \{1,\ldots,7\}$. Since $\{H_1,\ldots, H_7\}$ covers $e$ exactly four times, there are exactly three $T$-joins in $\{J_1,\ldots, J_7\}$ containing $e$. Together with $J$, the family of $T$-joins $\mathcal J$ covers $e$ exactly four times. Therefore, $\mathcal J$ covers every edge $e$ of  $\widehat{\, G\,}$ exactly four times.  

For an edge $e\in E(G)\backslash E(\widehat{\, G\,})$, it is covered by $\mathcal J$ at most eight times. Hence, $J$ covers every edge of $\widehat{\, G\,}$ exactly four times and other edges of $G$ at most eight times.  Therefore, 
\[\begin{aligned} 
|J| &\leq\frac 1 8 \Big (|J|+\sum\limits_{i=1}^7|E(J_i)|\Big) \leq\frac 1 8 \Big (4|\widehat{\,G\,}|+8(|E(G)|- |E(\widehat{\,G\,})|)\Big )\\
          &=|E(G)|-\frac 1 2 |E(\widehat{\,G\,})|.
\end{aligned}\]
This completes the proof. 
\end{proof}

For 2-edge-connected graphs, the following result is a direct corollary of Theorem~\ref{thm:T-join}.

\begin{corollary}\label{cor:T-join}
Let $G$ be a 2-edge-connected graph and $T$ be an even subset of vertices. Then $G$ has a $T$-join of size at most $|E(G)|/2$.
\end{corollary}

\noindent{\bf Remark.}  The bounds in Theorem~\ref{thm:T-join} and Corollary~\ref{cor:T-join} are sharp. For example, $G$ is a tree with all even-degree vertices having degree two  and $T$ is the set of all odd-degree vertices, or $G$ is an even cycle  
and $T$ consists of the two antipodal vertices.  \medskip

In the following, Theorem \ref{thm:T-join} and Corollary~\ref{cor:T-join} will be applied to prove technical lemmas on signed-circuit covers, which are particularly useful in the proofs of our main results. 

For a signed graph $(G,\sigma)$, denote the set of all negative edges of $(G,\sigma)$ by $E^-(G,\sigma)$ and the set of all positive edges by $E^+(G,\sigma)$. We always use $G^+$ to denote the subgraph of $G$ induced by all edges in $E^+(G,\sigma)$. Then $\widehat{G^+}$ stands for the maximal bridgeless subgraph of $G^+$.

\begin{lemma}\label{lem:shortest-cycle-edgesed}
Let $(G,\sigma)$ be a signed graph such that $G^+$ is connected.  
For any $S\subseteq E^-(G,\sigma)$, the signed graph $(G,\sigma)$ has an even-subgraph $H$ such that
\[S\subseteq H\subseteq G^+\cup S\quad \mbox{and} \quad |E(H)|\leq |E(G)|-\frac 1 2 |E(\widehat{G^+})|.\]
\end{lemma}

\begin{proof} 
Let $(G,\sigma)$ is a signed graph with  $G^+$ being connected. For any $S\subseteq E^-(G,\sigma)$, let $G[S]$
 be the subgraph $G[S]$ of $G$ induced by edges in $S$. Let $T$ be the set of all vertices of odd-degree in $G[S]$, which has an even number of vertices.
Let $J$ be a minimum $T$-join of $G^+$ with respect to  $T$. By Theorem~\ref{thm:T-join}, it follows that $|J|\le |E(G^+)|-\frac 1 2 |E(\widehat{G^+})|$. 

Let $H=G[S]\cup J$. Then $H$ is an even-subgraph which satisfies
\[S\subseteq H\subseteq G^+\cup S.\]
 Since $S\cap E(G^+)=\emptyset$, it follows
\[\begin{aligned} 
|E(H)| &= |E(G[S]\cup J)|=|S|+|J| \le |S|+ |E(G^+)|-\frac 1 2 |E(\widehat{G^+})|\\
          &\le |E(G)|-\frac  1 2 |E(\widehat{G^+})|.
\end{aligned}\]
This completes the proof.
\end{proof}

Now, we are going to define a function on signed graphs, which plays a key role in our proofs. Let $\mathcal B$ be the union of all barbells of $(G,\sigma)$. For any barbell $B\in \mathcal B$, recall that $\widehat{\,B\,}$ is the maximum bridgeless subgraph of $B$, which is the union of two negative cycles of $B$. Define
\[
\tau(G,\sigma)=\left\{ \begin{array}{ll}
|E(G)| & \mbox{if}\ \mathcal B=\emptyset; \\
\min\big \{|E(\widehat{\,B\,})|\, \big |\,B\in \mathcal B\big \} & otherwise.
 \end{array} \right.
\]  
The following result connects the function $\tau(G,\sigma)$ and signed-circuit cover, which is the key lemma to prove our main theorem.   \medskip

\begin{lemma}\label{lem:S-cover}
Let $(G,\sigma)$ be a 2-edge-connected  
signed graph such that $G^+$ is connected. If $S\subseteq E^-(G,\sigma)$ has an even number of edges, 
then $(G,\sigma)$ has  a family $\mathcal F$ of signed-circuits such that:\\
(i) $\mathcal F$ covers $S$ and all 2-edge-cuts containing an edge from $S$; \\
(ii) the length of $\mathcal F$ satisfies 
\[
\ell(\mathcal F)\leq\frac 1 2 \big (3|E(G)|-|E(\widehat{G^+})|-\tau(G,\sigma)\big );
\]
(iii) each negative loop of $S$ (if exists) is contained in exactly one barbell of $\mathcal F$.
\end{lemma}
\begin{proof}
By Lemma~\ref{lem:shortest-cycle-edgesed},  
$(G,\sigma)$ has an even-subgraph $H$ such that
\[ S\subseteq H\subseteq G^+\cup S\quad \mbox{and} \quad |E(H)|\leq |E(G)|-\frac 1 2 |E(\widehat{G^+})| .\]
Note that, for a 2-edge-cut $R$ of $G$, either $R\cap E(H)=\emptyset$ or $R\subset E(H)$ because $H$ is an even-subgraph.  
So all edges of 2-edge-cuts containing an edge from $S$ belong to $E(H)$. For (i), it suffices to show that $(G,\sigma)$ has a family of signed-circuits covering all edges of $H$. 

Since $H$ is an even-subgraph,  it has a cycle decomposition, i.e., its edges can be decomposed into edge-disjoint cycles. Choose a cycle decomposition $\mathcal C$ of $H$ such that  the number of positive cycles in $\mathcal C$ is maximum over all cycle decompositions of $H$. Then for any two negative cycles $C$ and $C'$ of $\mathcal C$, we have $|V(C)\cap V(C')|\le 1$. Otherwise $C\cup C'$ contains a positive cycle and can be decomposed into a family of cycles $\mathcal C'$, in which at least one cycle is positive. Then $\mathcal C'\cup (\mathcal C\backslash \{C,C'\})$ is a cycle decomposition of $H$ which has more positive cycles than $\mathcal C$, contradicting the choice of $\mathcal C$. For two negative cycles $C$ and $C'$ with exactly one common vertex, the union $C\cup C'$ is a short barbell of $(G,\sigma)$.
Let $\mathcal F=\{D_1,D_2,\ldots,D_m,C_1,C_2,\ldots,C_{2k-1},C_{2n}\}$ such that $D_i$  for $i\in \{1,\ldots,m\}$ is either a positive cycle of $\mathcal C$ or a short barbell consisting of two negative cycles of $\mathcal C$ with exactly one common vertex, and $C_i$  for $i\in \{1,\ldots  2n\}$ is a negative cycle of $\mathcal C$. (Note that $\mathcal F$ has an even number of negative cycle because $|S|$ is even.) By the above choices, we have $C_i\cap C_j=\emptyset$ for $i,j\in \{1,\ldots ,2n\}$ and $i\ne j$.

If $n=0$, then $\mathcal F$ is a signed-circuit cover of $H$ with length
\[\begin{aligned}
\ell (\mathcal F)  & =    |E(H)|\leq |E(G)|-\frac 1 2 |E(\widehat{G^+})|=\frac 1 2 \Big (3|E(G)| - |E(\widehat{G^+})|-|E(G)|\Big ) \\
 & \leq    \frac 1 2 \Big( 3|E(G)|-|E(\widehat {G^+})|-\tau(G,\sigma)\Big).
\end{aligned}
\]
The last inequality above follows from the fact that $\tau(G,\sigma)\le |E(G)|$. So $\mathcal F$ is a  family of signed-circuits of the type we seek. \medskip

So, in the following, assume that $n\ge 1$. Let $Q=\bigcup_{C_i\in \mathcal F} C_i$, and 
let $G'$ be the resulting graph obtained from $G$
by contracting each cycle $C_i$ for $i\in\{1,2,\ldots ,2n\}$
%, and each component of the subgraph induced by $\bigcup\limits_{j=1}^m E(D_j)$
to a single vertex $u_i$. 
%and deleting all edges from all $E(D_j)$ for $j\in\{1,2,\ldots, m\}$. 
Then $E(G')=E(G)\backslash (\bigcup_{C_i\in \mathcal F} E( C_i))=E(G)\backslash E(Q)$.

Let $U=\{u_1,u_2,\ldots, u_{2n}\}$ and $J$ be a minimum $T$-join of $G'$ with respect to $U$. In the graph $G'$, the subgraph induced by $J$ has $n$ edge-disjoint paths $P_1, P_2, \ldots, P_n$ such that each $P_k$ joins two distinct vertices  from $U$. By Corollary~\ref{cor:T-join}, it follows that 
\begin{equation}\label{eq:2}
|J|\le \frac 1 2 |E(G')|=\frac 1 2 (|E(G)|-|E(Q)|).
\end{equation}

Let $J_0$ be the subgraph of $G$ induced by edges in $J$. Then each endvertex $u_i$ of $n$ edge-disjoint paths of $J_0$ becomes a vertex $v_i\in C_i$ of $G$. Let $U'=\{v_1,v_2,...,v_{2n}\}$ where $v_i\in C_i$ for $i\in \{1,2,\ldots, 2n\}$. 
Note that, the edges of each $P_k$ with $k\in \{1,\ldots,n\}$ may induce a disconnected graph in $G$ consisting of subpaths of $P_k$. Then, for each path $P_k$, join all subpaths of $P_k\cap J_0$ 
 by using a segment from each cycle $C_i$ corresponding to an intermediate vertex $v_i\in U'$ of $P_k$, and the edge set  $J_1$ of the resulting subgraph is a $T$-join  of $G$ with respect to $U'$ satisfying $J_1\cap E(G')=J$ and $J_1\subseteq J\cup E(Q)$. 
 
Among all $T$-joins $J_1$ of $G$ with  respect to $U'$ such that $J_1\subseteq J\cup E(Q)$ and $J_1\cap E(G')=J$, let $J'$ be a such $T$-join with $|J'\cap E(H)|$ being minimum. If $J'$ induces a cycle, let $C$ be a such cycle. If $E(C)\subseteq E(Q)$, then $J'\backslash E(C)$ is still a $T$-join of $G$ with respect to $U'$ and has fewer edges than $J'$ which contradicts that $|J'\cap E(H)|$ is minimum. If $E(C)\backslash E(Q)\ne \emptyset$, then contracting edges of $C$ from $E(Q)$ generates a cycle of $J_0$ which contradicts that $J$ is a minimum $T$-join of $G'$. Hence $J'$ induces an acyclic subgraph of $G$.

\medskip

{\bf Claim:} {\sl  $|J'\cap E(Q)|\leq \frac 1 2 (|E(Q)|-\tau(G,\sigma))$.} \medskip

\noindent{\em Proof of the claim.} The subgraph induced by $J'$ in $G$ is acyclic and hence has at least two vertices of degree 1. Since $J'$ is a $T$-join of $G$ with respect to $\{v_1,\ldots, v_{2n}\}$. We may assume that $d_{J'}(v_1)=1$ and $d_{J'}(v_2)=1$ (relabelling if necessary). So $J'\cap E(C_i)=\emptyset $ for $i\in \{1,2\}$. Note that $C_1$ and $C_2$ are disjoint. Since $G$ is 2-edge-connected, $G$ has a path joining $C_1$ and $C_2$ and hence $G$ has a barbell containing $C_1$ and $C_2$. It follows that 
$\tau(G,\sigma) \le |E(C_1)|+|E(C_2)|$.  

Let $ Q'=Q\backslash (C_{1}\cup C_{2})$ and let $ J''=J'\oplus E(Q')$, the symmetric difference of $J'$ and $Q'$. 
Since $Q'\subseteq Q$ and $E(G')=E(G)\backslash E(Q)$, it follows that $J''$ is a $T$-join of $G$ with respect to 
$U'$ such that $|J''\cap E(G')|=|J'\cap E(G')|=|J\cap E(G')|$. By the choice of $J'$, we have 
$|J'\cap E(Q)|\leq | J'' \cap E(Q)|$. Furthermore, $J''\cap E(C_i)=J'\cap E(C_i)=\emptyset$ for $i\in \{1,2\}$.  
Therefore, $J'\cap E(Q)$ and $J''\cap E(Q)$ form a partition of $E(Q')$. Hence,
\[ \begin{aligned} 
 |J' \cap E(Q)|& \leq \frac 1 2 (|J'\cap E(Q)|+|J''\cap E(Q)|)=\frac 1 2 |E(Q')|\\
                     &=\frac 1 2 (|E(Q)|-(|E(C_1)|+|E(C_2)|)) \\
                     &\leq\frac 1 2 (|E(Q)|-\tau(G,\sigma)).
                     \end{aligned}\]
%The last inequality above holds because $\tau(G,\sigma)\le |E(C_1)|+|E(C_2)|$. 
This completes the proof of the claim.

\vskip 3mm

Note that the subgraph induced by $J'$ in $G$ has $n$ edge-disjoint paths joining $n$ pairs of vertices of $U'$, and these $n$ pairs of vertices form a partition of $U'$. So these $n$ edge-disjoint paths together with the negative cycles $C_i$'s containing the endvertices of the paths form $n$ barbells, denoted by
 $B_1,B_2,\ldots,B_n$. 
 %be the barbells constructed by $n$ edge-disjoint paths of $\mathcal P$ such that each path in $\mathcal P$ joins two negative cycles in $\mathcal F$.
Let $\mathcal F'=\{D_1,D_2,\ldots,D_m,B_1,B_2,\ldots,B_n\}$. Then $\mathcal F'$ is a signed-circuit cover of $H$. So the conclusion (i) follows. By the claim and (\ref{eq:2}), 
\[
\begin{aligned} 
|J'| &=|J' \cap E(G')|+|J'\cap E(Q)|=|J|+|J'\cap E(Q)|\\
     & \leq \frac 12 (|E(G)|-|E(Q)|)+\frac 1 2 (|E(Q)|-\tau(G,\sigma))\\
     &=\frac 1 2 (|E(G)|-\tau(G,\sigma)).
\end{aligned}
\]
So the length of $\mathcal F'$ satisfies
$$
\ell(\mathcal F')=|E(H)|+|J'|
\leq|E(H)|+\frac 12 (|E(G)|-\tau(G,\sigma))
\leq\frac 1 2 (3|E(G)|-|E(\widehat{G^+})|-\tau(G,\sigma)),
$$
which completes the proof of (ii). 

If $(G,\sigma)$ has a negative loop, then the negative loop is not contained in $J'$ by the minimality of $J'$. So a negative loop (if exists) is contained in exactly one barbell of $\mathcal F'$, and (iii) follows.
\end{proof}

\section{Shortest signed-circuit covers}

Let $(G,\sigma)$ be a signed graph. An edge cut $R$ of $(G,\sigma)$ is a minimal set of edges whose removal disconnects $G$.  A {\em switching operation $\zeta$} on $R$ is a mapping $\zeta: E(G)\to \{-1,1\}$ such that $\zeta (e)=-1$ if $e\in R$ and $\zeta(e)=1$, otherwise. Two signatures $\sigma$ and $\sigma'$ are {\em equivalent} if there exists an edge cut $R$ such that $\sigma (e)=\zeta (e) \cdot \sigma'(e)$ where $\zeta$ is the switching operation on $R$ \cite{TZ}. The {\em negativeness} of a signed graph $(G,\sigma)$ is the smallest number of negative edges over all equivalent signatures of $\sigma$, denoted by $\epsilon(G,\sigma)$. A signed graph is {\em balanced} if $\epsilon(G,\sigma)=0$. In other words, a balanced signed graph is equivalent to a graph (a signed graph without negative edges).  It is known that a signed graph $(G,\sigma)$ has a circuit cover if and only if $\epsilon(G,\sigma)\ne 1$ and every cut-edge of $G$ does not separate a balanced component, so-called {\em sign-bridgeless} (cf. \cite{CLLZ, mrrs}). The length of a shortest signed-circuit cover of a signed graph $(G,\sigma)$ is denoted by $\scc(G,\sigma)$. The following are some usefuly observations given in \cite{WY}.

\begin{observation}\label{ob:equ}
Let $(G,\sigma)$ be a flow-admissible signed graph and $\sigma'$ be an equivalent signature of $\sigma$.
Then $\scc(G,\sigma)=\scc(G,\sigma').$
\end{observation}

By Observation~\ref{ob:equ}, for shortest circuit cover problem, it is sufficient to consider all flow-admissible signed graphs $(G,\sigma)$ with $\epsilon(G,\sigma)$ negative edges.
 
%\begin{observation}\label{ob:1-negative-edge}
%Let $(G,\sigma)$ be a 2-edge-connected signed graph. Then $(G,\sigma)$ has a circuit cover if and only if $\epsilon(G,\sigma)\ne 1$.
%\end{observation}

 \begin{observation}\label{ob:cut}
Let $(G,\sigma)$ be a signed graph with $|E^-(G,\sigma)|=\epsilon(G,\sigma)$. Then every edge cut contains at most
half the number of edges being negative.
\end{observation}

By Observation~\ref{ob:cut}, $G^+$ is connected if $(G,\sigma)$ is connected and has minimum number of negative edges over all equivalent signatures, i.e., $|E^-(G,\sigma)|=\epsilon (G,\sigma)$.  Before proceed to prove our main result, we have the following result for 2-edge-connected signed graphs with even negativeness.

\begin{theorem}
Let $(G,\sigma)$ be a 2-edge-connected  signed graph with even negativeness. Then $(G,\sigma)$ has a signed-circuit cover with length less than $\frac 8 3 |E(G)|.$
\end{theorem}

\begin{proof} By Observation~\ref{ob:equ},  without loss of generality, we may assume that $(G,\sigma)$ has minimum number of negative edges, i.e., $|E^-(G,\sigma)|=\epsilon(G,\sigma) \equiv 0\pmod 2$. If $\epsilon(G,\sigma)=0$, the result follows from Theorem~\ref{thm:5/3}. So, in the following, assume that $|E^-(G,\sigma)|=\epsilon(G,\sigma)\ge 2$.

By Observation~\ref{ob:cut}, $G^+$ is connected. 
By Lemma~\ref{lem:S-cover}, $(G,\sigma)$ has a family of signed-circuits $\mathcal F_1$ covering all edges in $E(G)\backslash E(\widehat{G^+})$ with length
\[\ell(\mathcal F_1)\le \frac 1 2 (3|E(G)|-|E(\widehat{G^+})|-\tau(G,\sigma)) \le \frac 1 2 (3|E(G)|-|E(\widehat{G^+})|).\]
By Theorem~\ref{thm:5/3}, $\widehat{G^+}$ has a circuit cover $\mathcal F_2$ with length $\ell(\mathcal F_2)\le \frac 5 3 |E(\widehat{G^+})|$. Note that $\mathcal F=\mathcal F_1\cup \mathcal F_2$ is a signed-circuit cover of $(G,\sigma)$. 
Since $|E(\widehat{G^+})|\leq |E(G)|-|E^-(G,\sigma)|$, it follows that
\[\begin{aligned}
\ell(\mathcal F) & \leq \ell(\mathcal F_1)+\ell(\mathcal F_2)   \leq\frac 1 2 (3|E(G)|-|E(\widehat{G^+})|)+\frac 5 3 |E(\widehat{G^+})|\\
                        &=\frac 3 2 |E(G)| +\frac 7 6 |E(\widehat{G^+})|< \frac 3 2 |E (G)| +\frac 7 6 |E(G)| \\
                        &=\frac 8 3 |E(G)|.
\end{aligned}
\] This completes the proof. 
\end{proof}

%\begin{theorem}
%Let $(G,\sigma)$ be a 2-edge-connected  signed graph with even negativeness.
%Then there is a family of circuits $\mathcal F$ of $(G,\sigma)$
%covering $E(G)\backslash E(\widehat{G^+})$ such that
%$$\ell(\mathcal F)\leq \frac 1 2 (3|E(G)|-|E(\widehat{G^+})|)$$
%and every negative loop is contained in one barbell of $\mathcal F$.
%\end{theorem}
%\begin{proof}

%Choose $S=E^-(G,\sigma)$. Then $E^-(G,\sigma)\cup T(G,\sigma)=E(G)\backslash E(\widehat{G^+})$ and
%by Lemma~\ref{lem:S-cover}, there exists a circuit cover $\mathcal F$ covering
%$E(G)\backslash E(\widehat{G^+})$ such that
%$$\ell(\mathcal F)\leq\frac 1 2 (3|E(G)|-|E(\widehat{G^+})|-\tau(G,\sigma))\leq \frac 1 2 (3|E(G)|-|E(\widehat{G^+})|)$$
%and every negative loop is contained in one barbell of $\mathcal F$.

%\end{proof}

In the following, we are going to prove our main result, Theorem~\ref{thm:main}. First, we need a variation of Lemma~\ref{lem:S-cover} as follows.

\begin{lemma}\label{lem:t-members-2}
Let $(G,\sigma)$ be a 2-edge-connected flow-admissible signed graph with a negative loop $e$ such that $G^+$ is connected. Then, for %any negative loop $e$ and
any integer $t\in \{1,2\}$,
there is a family of circuits $\mathcal F_t$ of $(G,\sigma)$ such that:\\
(i) $\mathcal F_t$ covers all edges of $E(G)\backslash E(\widehat{G^+})$;\\
(ii) the length of $\mathcal F_t$ satisfies 
\[
\ell(\mathcal F)\leq2|E(G)|-\frac 12 |E(\widehat{G^+})|;
\] 
(iii) $e$ is contained in exactly $t$ barbells of $\mathcal F$.
\end{lemma}

\begin{proof} Since $(G,\sigma)$ is flow-admissible and $e$ is a negative loop, $(G-e,\sigma)$ has a negative edge $e'$ and hence has a negative cycle (for example, a cycle in $G^+\cup e'$ containing $e'$). Let $C_e$ be a negative cycle of $(G-e,\sigma)$ with minimum number of edges and let $B_e$ be a barbell consisting of $e$ and $C_e$ with minimum length.
 \medskip

\noindent{\bf Claim.} $|E(B_e)|\leq\frac 1 2 (|E(G)|+\tau(G,\sigma))$. \medskip

\noindent {\em Proof of the claim.}  Let $B$ be a barbell of $(G,\sigma)$ such that $|E(\widehat{\,B\,})|=\tau(G,\sigma)$.
Suppose $C_1$ and $C_2$ are two negative cycles contained in $B$.

The minimality $C_e$ implies that
$|E(C_e)|\leq \max\{ |E(C_1)|, |E(C_2)|\}$. Note that $1\le \min\{ |E(C_1)|, |E(C_2)|\}$. Therefore,
$$\tau(G,\sigma)\le |E(\widehat{B_e})|=|E(C_e)|+1\leq |E(C_1)|+|E(C_2)|=\tau(G,\sigma).$$
Hence, $|E(\widehat{B_e})|=\tau(G,\sigma)$. 
Since $(G,\sigma)$ is 2-edge-connected, there exist two edge-disjoint minimal paths $P_1$ and $P_2$, joining $e$ and $C_e$. So
\[\begin{aligned}
|E(B_e)| & \leq\frac 1 2 (|E(P_1\cup \widehat {B_e})|+|E(P_2\cup \widehat{B_e})|) \\
              & =\frac 1 2 \big ((|E(P_1)|+|E(P_2)|+|\widehat{B_e}|)+|\widehat{B_e}|\big ) \\
&\leq\frac 1 2 (|E(G)|+|E(\widehat{B_e})|)=\frac 1 2 (|E(G)|+\tau(G,\sigma)).
\end{aligned}
\]
This completes the proof of the claim. \medskip

If $|E^-(G,\sigma)|$ is even, by Lemma~\ref{lem:S-cover}, $(G,\sigma)$ has a family of signed-circuits
$\mathcal F_1$ which satisfies (i) covering $E^-(G,\sigma)$ and all 2-edge-cuts containing a negative edge, and hence covering all edges of $E(G)\backslash E(\widehat{G^+})$; (ii) having length
\[
\ell(\mathcal F_1)\leq \frac 1 2 (3|E(G)|-|E(\widehat{G^+})|-\tau(G,\sigma)) \le 2|E(G)|-\frac 1 2 |E(\widehat{G^+})|;
\]
and (iii) $e$ is contained in exactly one barbell of $\mathcal F_1$. So Lemma 3.4 follows if $t=1$. Now, assume that $t=2$. By the claim, let $\mathcal F_2=\mathcal F_1\cup \{B_e\}$ which is a family of signed-circuits satisfying (i) and having length 
\[\begin{aligned}
\ell(\mathcal F_2)& \le \ell(\mathcal F_1)+|E(B_e)| \\
                           & \leq\frac 1 2 \big (3|E(G)|-|E(\widehat{G^+})|-\tau(G,\sigma)\big )+\frac 1 2 (|E(G)|+\tau(G,\sigma))\\
                           &=2|E(G)|-\frac 1 2 |E(\widehat{G^+})|.
                           \end{aligned}\]
Then $\mathcal F_2$ is a family of signed-circuits of the type we seek. 

\vskip 3mm

In the following, assume that $|E^-(G,\sigma)|$ is odd.  Let $S_1=E^-(G,\sigma)\backslash e$ and let $S_2=E^-(G,\sigma)\backslash e'$ where $e'$ is a negative edge in $C_e$.
For each $S_t$ with $t\in \{1,2\}$, 
by Lemma~\ref{lem:S-cover},  $(G,\sigma)$ has a family of signed-circuits, denoted by $\mathcal F_{S_t}$,
which covers $S_t$ and all 2-edge-cuts containing an edge in $S_t$ and has length 
\[
\ell(\mathcal F_{S_t})\leq \frac 1 2 (3|E(G)|-|E(\widehat {G^+})|-\tau(G,\sigma)).
\]
Particularly, the loop $e$ is not covered by $\mathcal F_{S_1}$ but is contained in exactly one barbell of $\mathcal F_{S_2}$.

Let $\mathcal F_t=\mathcal F_{S_t}\cup\{B_e\}$ for $t\in \{1, 2\}$. Then $\mathcal F_t$ with $t\in \{1,2\}$ covers $E^-(G,\sigma)$ and all 2-edge-cuts containing an negative edge (a 2-edge-cut containing $e'$ is covered by $B_e$). Hence $\mathcal F_t$ covers all edges of $E(G)\backslash E(\widehat{G^+})$. By the claim, the length of $\mathcal F_t$ with $t\in \{1,2\}$ satisfies 
\[\begin{aligned}
\ell(\mathcal F_t)& =\ell(\mathcal F_{S_t})+|E(B_e)| \\
                        &\leq\frac 1 2\big (3|E(G)|-|E(\widehat{G^+})|-\tau(G,\sigma)\big)+\frac 1 2 (|E(G)|+\tau(G,\sigma)) \\
                        &=2|E(G)|-\frac 1 2 |E(\widehat{G^+})|.
\end{aligned}
\] 
Note that $e$ is contained in exactly $t$ barbells of $\mathcal F_t$ for $t\in \{1,2\}$. This completes the proof. 
\end{proof}

\begin{lemma}\label{lem:barbell}
Let $(G,\sigma)$ be a 2-edge-connected 
%flow-admissible 
signed graph such that $|E^-(G,\sigma)|\ge 2$ and $G^+$ is connected.
Then $(G,\sigma)$ has a  signed-circuit $D$ such that $D$ contains a negative edge in its cycle and
\[|E(D)|\leq \frac 1 2 (\tau(G,\sigma)+|E(G)|).\]
\end{lemma}

%\footnote{If $(G,\sigma)$ is 2-vertex-connected signed graph(i.e, sub-cubic signed graph),
%then $$|E(D)|\leq \frac{\tau(G,\sigma)+|V(G)|}{2}+1$$ and applying that, we have $$scc(G,\sigma)\leq\min\left
%\{\frac{8|E(G)|}{3}+\frac{|V(G)|}{2},\frac{23|E(G)|}{12}+\frac{3|V(G)|}{2}\right\}.$$
%In particular if $(G,\sigma)$ is a flow-admissible cubic signed graph, then $scc(G,\sigma)\leq \frac{35}{12}|E(G)|$.}

%\footnote{We may need to add $D$ contains a negative edge. See Theorem 4.3.}

\begin{proof}
First assume that $(G,\sigma)$ does not contain a barbell. Then $\tau(G,\sigma)=|E(G)|$. 
%The result holds trivially if $E^{-1}(G,\sigma)=\emptyset$. 
Let $e, e'\in E^-(G,\sigma)$ and $S=\{e, e'\}$. By Lemma~\ref{lem:shortest-cycle-edgesed}, there exists an even subgraph $H$ such that $S\subseteq E(H)\subseteq G^+\cup S$ and
\[ |E(H)|\leq |E(G)|-\frac 1 2 |E(\widehat{G^+})|.\]
Let $C_e$ and $C_{e'}$ be two cycle of $H$ containing $e$ and $e'$, respectively. Then either $C_e= C_{e'}$ or $|V(C_e\cap C_{e'})|\ge 2$. Otherwise, $(G,\sigma)$ has a barbell, which contradicts to the assumption. No matter $C_e= C_{e'}$ or $|V(C_e\cap C_{e'})|\ge 2$, $H$ has  a cycle containing both $e$ and $e'$, denoted by $D$. Since $D\subseteq H\subseteq G^+\cup S$, $D$ has exactly two negative edges and hence is a positive cycle (a signed-circuit) of $(G,\sigma)$. Furthermore,
\[|E(D)|\le |E(H)|\le  |E(G)|-\frac 1 2 |E(\widehat{G^+})|  \leq |E(G)|=\frac 1 2 (\tau(G,\sigma)+|E(G)|).\]

In the following, assume that $(G,\sigma)$ does have a barbell. If $(G,\sigma)$ itself is a short barbell, then let $D=(G,\sigma)$ and $\tau (G,\sigma)= |E(D)|$.  Hence $|E(D)|=|E(G)|=\frac 1 2 (\tau(G,\sigma) +|E(G)|)$ and the lemma holds trivially. Therefore, assume that $(G,\sigma)$ is not a barbell.  Then  $\tau(G,\sigma)<|E(G)|$.

Among all barbells $B$ of $(G,\sigma)$ with $|E(\widehat{\,B\,})|=\tau(G,\sigma)$, let $D$ be a such barbell with minimum number of edges. Then $D$ has a negative cycle which contains a negative edge. Let $P$ be the  path of $D$ joining the two negative cycles. By the minimality of $D$, the path $P$ has the shortest length among all minimal paths joining the two 
 cycles of $D$. Since $(G,\sigma)$ is 2-edge-connected, there exists two edge-disjoint minimal paths $P_1$ and $P_2$ joining the two negative cycles of $D$. Then $|E(D)|\le |E(\widehat{\,D\,})|+|E(P_i)|$ for $i\in \{1,2\}$.
By the minimality of $P$,
$$|E(D)|\le \frac 1 2 \Big ( 2 |E(\widehat{\,D\,})|+|E(P_1)|+|E(P_2)|\Big )
\leq \frac 1 2 \Big ( |E(\widehat {\,D\,})|+|E(G)|\Big )=\frac 1 2 (\tau(G,\sigma)+|E(G)|).$$
This completes the proof.
\end{proof}

%\noindent{\bf Remark.} The bound of the above lemma is sharp. For examples, $(G,\sigma)$ is a positive cycle or a signed graph consisting of two disjoint negative cycles joining by two edge-disjoint paths with common endvertices. \medskip

\begin{theorem}\label{thm:cover-bridge}
Let $(G,\sigma)$ be a flow-admissible signed graph. Then $(G,\sigma)$ has a family of circuits  $\mathcal F$ 
covering all edges in $E(G)\backslash E(\widehat{G^+})$ such that every negative loop is covered at most twice and 
its length satisfies
\[\ell(\mathcal F)\leq 2|E(G)|-\frac 1 2 |E(\widehat{G^+})|+b(G,\sigma),\]
where $b(G,\sigma)$ is the number of cut-edges of $(G,\sigma)$.
\end{theorem}

\begin{proof} Note that, the theorem is ture if it holds for every connected component of $(G,\sigma)$. So, without loss of generality, assume that $(G,\sigma)$ is connected. 
By Observation~\ref{ob:equ}, we may assume that $(G,\sigma)$ has $|E^-(G,\sigma)|=\epsilon(G,\sigma)$. If $\epsilon(G,\sigma)=0$, then $E(G)\backslash E(G^+)=\emptyset$ and the result holds trivially. So assume that $\epsilon(G,\sigma)\ne 0$. Since $(G,\sigma)$ is flow-admissible, it follows that $|E^-(G,\sigma)|\ge 2$.  \medskip

We apply induction on $b(G,\sigma)$ to prove the theorem. First, we verify the base case $b(G,\sigma)=0$. In other words, $(G,\sigma)$ is 2-edge-connected. 

If $\epsilon(G,\sigma)$ is even,  by Lemma~\ref{lem:S-cover}, $(G,\sigma)$ has a family of signed-circuits $\mathcal F$ covering all 2-edge-cut containing a negative edge such that each negative loop is covered exactly once and
\[\ell(\mathcal F)\le \frac 1 2 (3|E(G)|-|E(\widehat{G^+})| -\tau(G,\sigma))\le 2|E(G)|-\frac 1 2 \tau(G,\sigma).\] Then $\mathcal F$ is a family of signed-circuits of the type desired. So assume that $\epsilon(G,\sigma)$ is odd and $\epsilon(G,\sigma)\ge 3$.
By Lemma~\ref{lem:barbell}, $(G,\sigma)$ has a signed-circuit $D$ such that $D$ has a cycle containing a negative edge and 
\[|E(D)|\leq \frac 1 2 (\tau(G,\sigma)+|E(G)|).\]
Let $e$ be a negative edge which is contained in a cycle of $D$, % such that $T(e)\cup\{e\}\subseteq D$. 
%\footnote{Requiring $D$ with a negative edge.}
and let $S=E^-(G,\sigma)\backslash e$. By Lemma~\ref{lem:S-cover}, $(G,\sigma)$ has a family of signed-circuits $\mathcal F'$ covering $S$ and all 2-edge-cuts containing an edge from $S$ such that  each negative loop of $S$ is covered exactly once and
\[\ell(\mathcal F')\leq \frac 1 2 \big(3|E(G)|-|E(\widehat{G^+})|-\tau(G,\sigma)\big).\]
Let $\mathcal F=\mathcal F'\cup\{D\}$. Since $S\cup \{e\}=E^-(G,\sigma)$,  it follows from Observation~\ref{ob:cut} that every positive edge in $E(G)\backslash E(\widehat{G^+})$ is contained in a 2-edge-cut which contains an edge from $E^-(G,\sigma)$. Therefore,
$\mathcal F$ is a family of circuits of $(G,\sigma)$ covering all edges of $E(G)\backslash E(\widehat{G^+})$ such that
\[\begin{aligned} 
\ell(\mathcal F) & =\ell(\mathcal F')+|E(D)| \\
                         &\leq\frac 1 2 \big (3|E(G)|-|E(\widehat{G^+})|-\tau(G,\sigma)\big )+\frac 1 2 (\tau(G,\sigma)+|E(G)|)\\
                         &= 2|E(G)|-\frac 1 2 |E(\widehat{G^+})|.
                         \end{aligned}\]
Note that every negative loop is covered by exactly one barbell of $\mathcal F'$, and $D$ is a signed-circuit. Hence $\mathcal F$ covers every negative loop at most twice. So the theorem holds if $b(G,\sigma)=0$.  \medskip

So, in the following, assume that $b(G,\sigma)\ne 0$ and the theorem holds for all flow-admissible signed graph with at most $b(G,\sigma)-1$ cut-edges.

Since $b(G,\sigma)\ne 0$, the graph $G$ has a cut-edge. Let $uv$ be a cut-edge of $G$ such that $G\backslash uv$ consists of two components $Q_1$ and $Q_2$,  one of which, say $Q_2$, contains no cut-edges. Without loss of generaility, assume that $u\in V(Q_1)$ and $v\in V(Q_2)$. Since  $(G,\sigma)$ is flow-admissible, both $Q_1$ and $Q_2$ are not balanced and hence contain a negative edge. Hence $Q_2$ is either a negative loop or 2-edge-connected. 

Let $(G_1,\sigma)$ be the resulting signed graph constructed from  $(Q_1,\sigma)$  by adding a negative loop $e_1$ attached to $u$. Then $(G_1,\sigma)$ is flow-admissible and $b(G_1,\sigma)=b(G,\sigma)-1<b(G,\sigma)$. 
By induction hypothesis, $(G_1,\sigma)$ has  a family of signed-circuits $\mathcal F_1$
covering all edges in $E(G_1)\backslash E(\widehat{G_1^+})$ such that
\[\ell(\mathcal F_1)\leq 2|E(G_1)|-\frac 1 2 |E(\widehat{G_1^+})|+b(G_1,\sigma)\]
and every negative loop is contained in at most two barbells of $\mathcal F_1$. Assume that $e_1$ is contained in $t$ barbells of $\mathcal F_1$ with $t\in \{1,2\}$. 

Let $(G_2,\sigma)$ be the resulting graph constructed from  $(Q_2,\sigma)$ by attaching a negative loop $e_2$ to $v$. Then $(G_2,\sigma)$ is flow-admissible and 2-edge-connected. 
%Recall that $(Q,\sigma)$ contains a negative edge. Hence $(G_2,\sigma)$ is flow-admissible and 2-edge-connected.
By Lemma~\ref{lem:t-members-2}, $(G_2,\sigma)$ has a family of signed-circuits $\mathcal F_2$
covering all edges in $E(G_2)\backslash E(\widehat{G_2^+})$ such that
$$\ell(\mathcal F_2)\leq 2|E(G_2)|-\frac 1 2 |E(\widehat{G_2^+})|$$
and $e_2$ is contained in exactly $t$ barbells of $\mathcal F_2$.

Let $B'_i$ with $1\le i\le t$ be the barbells of $\mathcal F_1$ containing $e_1$, and $B_i''$ with $1\le i\le t$ be the barbells of $\mathcal F_2$ containing $e_2$. (Note that, $t\le 2$.) Let $B_i=(B_i'\backslash e_1)\cup (B_i''\backslash e_2) \cup uv$ which is a barbell of $(G,\sigma)$ for $1\le i \le t$, and let $\mathcal F=(\mathcal F_1\backslash \{B_i'| 1\le i\le t\})\cup (\mathcal F_2\backslash \{B_i''| 1\le i\le t\}) \cup \{B_i| 1\le i\le t\}$. Since $|E(B_i)|=|E(B_i')|+|E(B_i'')|-1$ for each $i$ with $1\le i\le t$. Then
\begin{equation}\label{eq:3}
\ell(\mathcal F)=\ell(\mathcal F_1)+\ell(\mathcal F_2) -t \le \big(2|E(G_1)| -\frac 1 2 |E(\widehat{G_1^+})|+b(G_1,\sigma)\big) +\big(2|E(G_2)|-\frac 1 2 |E(\widehat{G_2^+})|\big )-t.
\end{equation}

%If $t=1$, let $B_1$ be the unique barbell in  $\mathcal F_1$ containing $e_1$ and  $B_2$ is the unique barbell of $\mathcal F_2$ containing $e_2$. Let $B=(B_1\backslash e_1)\cup (B_2\backslash e_2) \cup uv$ and let $\mathcal F=(\mathcal F_1\backslash B_1)\cup (\mathcal F_2\backslash B_2) \cup B$. 
%Then $|E(B)|=|E(B_1)|+|E(B_2)|-1$, and further
%\begin{equation}\label{ineq-1}
%\ell(\mathcal F)=\ell(\mathcal F_1)+\ell(\mathcal F_2)-1\leq
%2|E(G_1)|-\frac 1 2 |E(\widehat{G_1^+})|+|B(G_1,\sigma)|+2|E(G_2)|-\frac 1 2 |E(\widehat{G_2^+})| -1.
%\end{equation}

%If $t=2$, then let $B_1^1$ and $B_1^2$ be two barbells in $\mathcal F_1$ containing $e_1$ and let $B_2^1$ and $B_2^2$ be two barbells of $\mathcal F_2$ containing $e_2$. Let $B_i=(B_1^i\backslash e_1)\cup (B_2^i\backslash e_2)\cup uv$ for $i=1$ and 2. Then $|E(B_i)|=|E(B_1^i)|+|E(B_2^i)|-1$ for $i=1$ and 2. Further, 
%\begin{equation}\label{ineq-2}
%\ell(\mathcal F)=\ell(\mathcal F_1)+\ell(\mathcal F_2)-1\leq
%2|E(G_1)|-\frac 1 2 |E(\widehat{G_1^+})|+|B(G_1,\sigma)|+2|E(G_2)|-\frac 1 2 |E(\widehat{G_2^+})| -2.
%\end{equation}

Since $\widehat{G_1^+}\cup \widehat{G_2^+}=\widehat{G^+}$ and $\widehat{G_1^+}\cap \widehat {G_2^+}=\emptyset$, we have 
$|E(\widehat {G_1^+})|+|E(\widehat{G_2^+})|=|E(\widehat{G^+})|$. 
Note that $b(G_1,\sigma)=b(G,\sigma)-1$ and $|E(G_1)|+|E(G_2)|=|E(G)|+1$.
It follows from  (\ref{eq:3}) that 
\[\ell(\mathcal F)\leq 2(|E(G)|+1)-\frac 1 2 |E(\widehat{G^+})|+(b(G,\sigma)-1)-t\le  2|E(G)|-\frac 1 2 |E(\widehat{G^+})|+b(G,\sigma).\]
Note that a negative loop of $(G,\sigma)$ (if exists) is contained in one or two barbells of either $\mathcal F_1$ or $\mathcal F_2$ (but not both). Hence a negative loop of $(G,\sigma)$ is contained in either one or two barbells in $\mathcal F$. This completes the proof.
\end{proof}

Now, we are ready to prove our main result, Theorem~\ref{thm:main}. \medskip

\noindent{\bf Proof of Theorem~\ref{thm:main}.} Let $(G,\sigma)$ be a flow-admissible signed graph.  By observation~\ref{ob:equ}, we may assume that $|E^-(G,\sigma)|=\epsilon(G,\sigma)$. If $\epsilon(G,\sigma)=0$, then the results follows directly from Theorem~\ref{thm:5/3}. So assume that $\epsilon(G,\sigma)\ge 2$. By  Observation~\ref{ob:cut},  $G^+$ is connected.
% $E^-(G,\sigma)\cup T(G,\sigma)\cup B(G,\sigma)=E(G)\backslash E(\widehat{G^+})$.  Hence
%$|E(\widehat{G^+})|\leq |E(G)|-|B(G,\sigma)|-|E^-(G,\sigma)|$, i.e., $|E(\widehat{G^+})|+|B(G,\sigma)|+|E^-(G,\sigma)|\le |E(G)|$. 

By Theorem~\ref{thm:cover-bridge}, $(G,\sigma)$ has a family of signed-circuits $\mathcal F_1$ covering all edges in $E(G)\backslash E(\widehat{G^+})$ with length
\[\ell(\mathcal F_1)\le 2|E(G)|-\frac 1 2 |E(\widehat {G^+})| +b(G,\sigma),\]
where $b(G,\sigma)$ is the number of cut-edges of $(G,\sigma)$. 
By Theorem~\ref{thm:5/3}, the subgraph $\widehat{G^+}$ has a signed-circuit cover $\mathcal F_2$ with length $\ell(\mathcal F_2) \le \frac 5 3 |E(\widehat{G^+})|$. Therefore, $\mathcal F=\mathcal F_1\cup \mathcal F_2$ is a signed-circuit cover of $(G,\sigma)$,  and
\[\begin{aligned} 
 \ell (\mathcal F) &=\ell(\mathcal F_1)+\ell(\mathcal F_2)  \leq 2|E(G)|-\frac 1 2 |E(\widehat{G^+})|+b(G,\sigma) + \frac 5 3 |E(\widehat{G^+})|\\
                     & \leq 2|E(G)| +\frac 7 6 |E(\widehat{G^+})|+b(G,\sigma)\\
                     & < 2|E(G)| +\frac 7 6  |E(G)| =\frac{19} 6 |E(G)|.
\end{aligned}\]
This completes the proof of Theorem~\ref{thm:main}. \qed

\medskip
\noindent {\bf Remark.} Let  $(P_{10},\sigma)$ be the signed graph with $P_{10}$ being the Petersen graph and $E^-(P_{10}, \sigma)$ inducing a 5-cycle. 
M\'a\v{c}ajov\'a et. al. \cite{mrrs} show that a shortest circuit cover of $(P_{10}, \sigma)$ has length exactly $\frac 5 3 |E(P_{10})|$. The optimal upper bound for the shortest signed-circuit cover remains to be investigated. 
 
%%-----------------------------------------------------------------------

\end{document}